\documentclass{amsart}
\usepackage{amsfonts}

\newtheorem{theorem}{Theorem}
\theoremstyle{plain}

\newtheorem{corollary}{Corollary}

\newtheorem{lemma}{Lemma}

\newtheorem{remark}{Remark}

\numberwithin{equation}{section}
\input{tcilatex}

\begin{document}
\title[On New Refinements and Reverses of Young's Operator Inequality]{On
New Refinements and Reverses of Young's Operator Inequality}
\author{S. S. Dragomir$^{1,2}$}
\address{$^{1}$Mathematics, College of Engineering \& Science\\
Victoria University, PO Box 14428\\
Melbourne City, MC 8001, Australia.}
\email{sever.dragomir@vu.edu.au}
\urladdr{http://rgmia.org/dragomir}
\address{$^{2}$School of Computer Science \& Applied Mathematics, University
of the Witwatersrand, Private Bag 3, Johannesburg 2050, South Africa}

\begin{abstract}
In this paper we obtain some new refinements and reverses of Young's
operator inequality. Extensions for convex functions of operators are also
provided.
\end{abstract}

\maketitle

\section{Introduction}

Throughout this paper $A$ and $B$ are positive operators on a complex
Hilbert space $\left( H,\left\langle \cdot ,\cdot \right\rangle \right) .$
We use the following notations for operators 
\begin{equation*}
A\nabla _{\nu }B:=\left( 1-\nu \right) A+\nu B,\text{ the \textit{weighted
arithmetic mean}}
\end{equation*}%
and%
\begin{equation*}
A\sharp _{\nu }B:=A^{1/2}\left( A^{-1/2}BA^{-1/2}\right) ^{\nu }A^{1/2},%
\text{ the \textit{weighted geometric mean.}}
\end{equation*}%
When $\nu =\frac{1}{2}$ we write $A\nabla B$ and $A\sharp B$ for brevity,
respectively.

The famous \textit{Young inequality} for scalars says that if $a,b>0$ and $%
\nu \in \lbrack 0,1],$ then%
\begin{equation}
a^{1-\nu }b^{\nu }\leq \left( 1-\nu \right) a+\nu b  \label{e.1.1}
\end{equation}%
with equality if and only if $a=b$. The inequality (\ref{e.1.1}) is also
called $\nu $-\textit{weighted arithmetic-geometric mean inequality}.

We recall that \textit{Specht's ratio} is defined by \cite{S} 
\begin{equation}
S\left( h\right) :=\left\{ 
\begin{array}{l}
\frac{h^{\frac{1}{h-1}}}{e\ln \left( h^{\frac{1}{h-1}}\right) }\text{ if }%
h\in \left( 0,1\right) \cup \left( 1,\infty \right) \\ 
\\ 
1\text{ if }h=1.%
\end{array}%
\text{ }\right.  \label{S}
\end{equation}%
It is well known that $\lim_{h\rightarrow 1}S\left( h\right) =1,$ $S\left(
h\right) =S\left( \frac{1}{h}\right) >1$ for $h>0,$ $h\neq 1$. The function
is decreasing on $\left( 0,1\right) $ and increasing on $\left( 1,\infty
\right) .$

The following inequality provides a refinement and a multiplicative reverse
for Young's inequality

\begin{equation}
S\left( \left( \frac{a}{b}\right) ^{r}\right) a^{1-\nu }b^{\nu }\leq \left(
1-\nu \right) a+\nu b\leq S\left( \frac{a}{b}\right) a^{1-\nu }b^{\nu },
\label{e.1.2}
\end{equation}%
where $a,b>0$, $\nu \in \lbrack 0,1],$ $r=\min \left\{ 1-\nu ,\nu \right\} $.

The second inequality in (\ref{e.1.2}) is due to Tominaga \cite{T} while the
first one is due to Furuichi \cite{F}.

The operator version is as follows \cite{F}, \cite{T} :

\begin{theorem}
\label{t.A}For two positive operators $A,$ $B$ and positive real numbers $m,$
$m^{\prime },$ $M,$ $M^{\prime }$ satisfying the following conditions (i) or
(ii):

(i) $0<m^{\prime }I\leq A\leq mI<MI\leq B\leq M^{\prime }I;$

(ii) $0<m^{\prime }I\leq B\leq mI<MI\leq A\leq M^{\prime }I;$

we have%
\begin{equation}
S\left( h^{r}\right) A\sharp _{\nu }B\leq A\nabla _{\nu }B\leq S\left(
h\right) A\sharp _{\nu }B  \label{e.1.2.1}
\end{equation}%
where $h:=\frac{M}{m},$ $h^{\prime }:=\frac{M^{\prime }}{m^{\prime }}$ and $%
\nu \in \left[ 0,1\right] .$
\end{theorem}

We consider the \textit{Kantorovich's constant }defined by 
\begin{equation}
K\left( h\right) :=\frac{\left( h+1\right) ^{2}}{4h},\text{ }h>0.  \label{K}
\end{equation}%
The function $K$ is decreasing on $\left( 0,1\right) $ and increasing on $%
\left[ 1,\infty \right) ,$ $K\left( h\right) \geq 1$ for any $h>0$ and $%
K\left( h\right) =K\left( \frac{1}{h}\right) $ for any $h>0.$

The following multiplicative refinement and reverse of Young inequality in
terms of Kantorovich's constant\textit{\ }holds%
\begin{equation}
K^{r}\left( \frac{a}{b}\right) a^{1-\nu }b^{\nu }\leq \left( 1-\nu \right)
a+\nu b\leq K^{R}\left( \frac{a}{b}\right) a^{1-\nu }b^{\nu }  \label{e.1.3}
\end{equation}%
where $a,b>0$, $\nu \in \lbrack 0,1],$ $r=\min \left\{ 1-\nu ,\nu \right\} $
and $R=\max \left\{ 1-\nu ,\nu \right\} .$

The first inequality in (\ref{e.1.3}) was obtained by Zou et al. in \cite%
{ZSF} while the second by Liao et al. \cite{LWZ}.

The operator version is as follows \cite{ZSF}, \cite{LWZ}:

\begin{theorem}
\label{t.B}For two positive operators $A,$ $B$ and positive real numbers $m,$
$m^{\prime },$ $M,$ $M^{\prime }$ satisfying the following conditions (i) or
(ii):

(i) $0<m^{\prime }I\leq A\leq mI<MI\leq B\leq M^{\prime }I;$

(ii) $0<m^{\prime }I\leq B\leq mI<MI\leq A\leq M^{\prime }I;$

we have%
\begin{equation}
K^{r}\left( h\right) A\sharp _{\nu }B\leq A\nabla _{\nu }B\leq K^{R}\left(
h\right) A\sharp _{\nu }B  \label{e.1.3.1}
\end{equation}%
where $h:=\frac{M}{m},$ $h^{\prime }:=\frac{M^{\prime }}{m^{\prime }}$, $\nu
\in \left[ 0,1\right] $ $r=\min \left\{ 1-\nu ,\nu \right\} $ and $R=\max
\left\{ 1-\nu ,\nu \right\} .$
\end{theorem}

Kittaneh and Manasrah \cite{KM1}, \cite{KM2} provided a refinement and an
additive reverse for Young inequality as follows:%
\begin{equation}
r\left( \sqrt{a}-\sqrt{b}\right) ^{2}\leq \left( 1-\nu \right) a+\nu
b-a^{1-\nu }b^{\nu }\leq R\left( \sqrt{a}-\sqrt{b}\right) ^{2}  \label{KM}
\end{equation}%
where $a,b>0$, $\nu \in \lbrack 0,1],$ $r=\min \left\{ 1-\nu ,\nu \right\} $
and $R=\max \left\{ 1-\nu ,\nu \right\} .$ The case $\nu =\frac{1}{2}$
reduces (\ref{KM}) to an identity.

For some operator versions of (\ref{KM}) see \cite{KM1} and \cite{KM2}.

In the recent paper \cite{D}\ we obtained the following reverses of Young's
inequality as well:%
\begin{equation}
0\leq \left( 1-\nu \right) a+\nu b-a^{1-\nu }b^{\nu }\leq \nu \left( 1-\nu
\right) \left( a-b\right) \left( \ln a-\ln b\right)   \label{e.2.10}
\end{equation}%
and%
\begin{equation}
1\leq \frac{\left( 1-\nu \right) a+\nu b}{a^{1-\nu }b^{\nu }}\leq \exp \left[
4\nu \left( 1-\nu \right) \left( K\left( \frac{a}{b}\right) -1\right) \right]
,  \label{e.2.11}
\end{equation}%
where $a,$ $b>0$, $\nu \in \lbrack 0,1].$

It has been shown in \cite{D} that there is no ordering for the upper bounds
of the quantity $\left( 1-\nu \right) a+\nu b-a^{1-\nu }b^{\nu }$ as
provided by the inequalities (\ref{KM}) and (\ref{e.2.10}). The same
conclusion is true for the upper bounds of the quantity $\frac{\left( 1-\nu
\right) a+\nu b}{a^{1-\nu }b^{\nu }}$ incorporated in the inequalities (\ref%
{e.1.2}), (\ref{e.1.3}) and (\ref{e.2.11}).

By the use of two new refinements and reverses of Young's inequality we
establish in this paper several other operators inequalities that are
similar to those from above. Extensions for convex functions of operators
with some examples are also provided.

\section{Some Preliminary Results}

We have the following result:

\begin{lemma}
\label{l.2.1}Let $f:I\subset \mathbb{R}\rightarrow \mathbb{R}$ be a twice
differentiable function on the interval $\mathring{I}$, the interior of $I$.
If there exists the constants $d,$ $D$ such that%
\begin{equation}
d\leq f^{\prime \prime }\left( t\right) \leq D\text{ for any }t\in \mathring{%
I},  \label{e.2.0}
\end{equation}%
then 
\begin{align}
\frac{1}{2}\nu \left( 1-\nu \right) d\left( b-a\right) ^{2}& \leq \left(
1-\nu \right) f\left( a\right) +\nu f\left( b\right) -f\left( \left( 1-\nu
\right) a+\nu b\right)  \label{e.2.1} \\
& \leq \frac{1}{2}\nu \left( 1-\nu \right) D\left( b-a\right) ^{2}  \notag
\end{align}%
for any $a,$ $b\in \mathring{I}$ and $\nu \in \left[ 0,1\right] .$

In particular, we have%
\begin{equation}
\frac{1}{8}\left( b-a\right) ^{2}d\leq \frac{f\left( a\right) +f\left(
b\right) }{2}-f\left( \frac{a+b}{2}\right) \leq \frac{1}{8}\left( b-a\right)
^{2}D,  \label{e.2.1.1}
\end{equation}%
for any $a,$ $b\in \mathring{I}$.

The constant $\frac{1}{8}$ is best possible in both inequalities in (\ref%
{e.2.1.1}).
\end{lemma}

\begin{proof}
We consider the auxiliary function $f_{D}:I\subset \mathbb{R}\rightarrow 
\mathbb{R}$ defined by $f_{D}\left( x\right) =\frac{1}{2}Dx^{2}-f\left(
x\right) .$ The function $f_{D}$ is differentiable on $\mathring{I}$ and $%
f_{D}^{\prime \prime }\left( x\right) =D-f^{\prime \prime }\left( x\right)
\geq 0,$ showing that $f_{D}$ is a convex function on $\mathring{I}.$

By the convexity of $f_{D}$ we have for any $a,$ $b\in \mathring{I}$ and $%
\nu \in \left[ 0,1\right] $ that%
\begin{align*}
0& \leq \left( 1-\nu \right) f_{D}\left( a\right) +\nu f_{D}\left( b\right)
-f_{D}\left( \left( 1-\nu \right) a+\nu b\right) \\
& =\left( 1-\nu \right) \left( \frac{1}{2}Da^{2}-f\left( a\right) \right)
+\nu \left( \frac{1}{2}Db^{2}-f\left( b\right) \right) \\
& -\left( \frac{1}{2}D\left( \left( 1-\nu \right) a+\nu b\right)
^{2}-f_{D}\left( \left( 1-\nu \right) a+\nu b\right) \right) \\
& =\frac{1}{2}D\left[ \left( 1-\nu \right) a^{2}+\nu b^{2}-\left( \left(
1-\nu \right) a+\nu b\right) ^{2}\right] \\
& -\left( 1-\nu \right) f\left( a\right) -\nu f\left( b\right) +f_{D}\left(
\left( 1-\nu \right) a+\nu b\right) \\
& =\frac{1}{2}\nu \left( 1-\nu \right) D\left( b-a\right) ^{2}-\left( 1-\nu
\right) f\left( a\right) -\nu f\left( b\right) +f_{D}\left( \left( 1-\nu
\right) a+\nu b\right) ,
\end{align*}%
which implies the second inequality in (\ref{e.2.1}0).

The first inequality follows in a similar way by considering the auxiliary
function $f_{d}:I\subset \mathbb{R}\rightarrow \mathbb{R}$ defined by $%
f_{D}\left( x\right) =f\left( x\right) -\frac{1}{2}dx^{2}$ that is twice
differentiable and convex on $\mathring{I}$.

If we take $f\left( x\right) =x^{2},$ then (\ref{e.2.0}) holds with equality
for $d=D=2$ and (\ref{e.2.1.1}) reduces to an equality as well.
\end{proof}

If $D>0,$ the second inequality in (\ref{e.2.1}) is better than the
corresponding inequality obtained by Furuichi and Minculete in \cite{FM} by
applying Lagrange's theorem two times. They had instead of $\frac{1}{2}$ the
constant $1$. Our method also allowed to obtain, for $d>0,$ a lower bound
that can not be established by Lagrange's theorem method employed in \cite%
{FM}.

We have:

\begin{theorem}
\label{t.2.1}For any $a,$ $b>0$ and $\nu \in \left[ 0,1\right] $ we have%
\begin{align}
\frac{1}{2}\nu \left( 1-\nu \right) \left( \ln a-\ln b\right) ^{2}\min
\left\{ a,b\right\} & \leq \left( 1-\nu \right) a+\nu b-a^{1-\nu }b^{\nu }
\label{e.2.13} \\
& \leq \frac{1}{2}\nu \left( 1-\nu \right) \left( \ln a-\ln b\right)
^{2}\max \left\{ a,b\right\}   \notag
\end{align}%
and%
\begin{align}
\exp \left[ \frac{1}{2}\nu \left( 1-\nu \right) \frac{\left( b-a\right) ^{2}%
}{\max^{2}\left\{ a,b\right\} }\right] & \leq \frac{\left( 1-\nu \right)
a+\nu b}{a^{1-\nu }b^{\nu }}  \label{e.2.14} \\
& \leq \exp \left[ \frac{1}{2}\nu \left( 1-\nu \right) \frac{\left(
b-a\right) ^{2}}{\min^{2}\left\{ a,b\right\} }\right] .  \notag
\end{align}
\end{theorem}

\begin{proof}
If write the inequality (\ref{e.2.1}) for the convex function $f:\mathbb{%
R\rightarrow }\left( 0,\infty \right) ,$ $f\left( x\right) =\exp \left(
x\right) ,$ then we have%
\begin{align}
& \frac{1}{2}\nu \left( 1-\nu \right) \left( x-y\right) ^{2}\min \left\{
\exp x,\exp y\right\}  \label{e.2.12} \\
& \leq \left( 1-\nu \right) \exp \left( x\right) +\nu \exp \left( y\right)
-\exp \left( \left( 1-\nu \right) x+\nu y\right)  \notag \\
& \leq \frac{1}{2}\nu \left( 1-\nu \right) \left( x-y\right) ^{2}\max
\left\{ \exp x,\exp y\right\}  \notag
\end{align}%
for any $x,$ $y\in \mathbb{R}$ and $\nu \in \left[ 0,1\right] .$

Let $a,$ $b>0.$ If we take $x=\ln a,$ $y=\ln b$ in (\ref{e.2.12}), then we
get the desired inequality (\ref{e.2.13}).

Now, if we write the inequality (\ref{e.2.1}) for the convex function $%
f:\left( 0,\infty \right) \rightarrow \mathbb{R}$, $f\left( x\right) =-\ln
x, $ then we get for any $a,$ $b>0$ and $\nu \in \left[ 0,1\right] $ that%
\begin{align}
\frac{1}{2}\nu \left( 1-\nu \right) \frac{\left( b-a\right) ^{2}}{%
\max^{2}\left\{ a,b\right\} }& \leq \ln \left( \left( 1-\nu \right) a+\nu
b\right) -\left( 1-\nu \right) \ln a-\nu \ln b  \label{e.2.12.a} \\
& \leq \frac{1}{2}\nu \left( 1-\nu \right) \frac{\left( b-a\right) ^{2}}{%
\min^{2}\left\{ a,b\right\} }.  \notag
\end{align}
\end{proof}

The second inequalities in (\ref{e.2.13}) and (\ref{e.2.14}) are better than
the corresponding results obtained by Furuichi and Minculete in \cite{FM}
where instead of constant $\frac{1}{2}$ they had the constant $1.$

Now, since%
\begin{equation*}
\frac{\left( b-a\right) ^{2}}{\min^{2}\left\{ a,b\right\} }=\left( \frac{%
\max \left\{ a,b\right\} }{\min \left\{ a,b\right\} }-1\right) ^{2}\text{
and }\frac{\left( b-a\right) ^{2}}{\max^{2}\left\{ a,b\right\} }=\left( 
\frac{\min \left\{ a,b\right\} }{\max \left\{ a,b\right\} }-1\right) ^{2},
\end{equation*}%
then (\ref{e.2.14}) can also be written as:%
\begin{align}
& \exp \left[ \frac{1}{2}\nu \left( 1-\nu \right) \left( 1-\frac{\min
\left\{ a,b\right\} }{\max \left\{ a,b\right\} }\right) ^{2}\right]
\label{e.2.12.b} \\
& \leq \frac{\left( 1-\nu \right) a+\nu b}{a^{1-\nu }b^{\nu }}  \notag \\
& \leq \exp \left[ \frac{1}{2}\nu \left( 1-\nu \right) \left( \frac{\max
\left\{ a,b\right\} }{\min \left\{ a,b\right\} }-1\right) ^{2}\right]  \notag
\end{align}%
for any $a,$ $b>0$ and $\nu \in \left[ 0,1\right] .$

\begin{remark}
\label{r.2.1}For $\nu =\frac{1}{2}$ we get the following inequalities of
interest%
\begin{equation}
\frac{1}{8}\left( \ln a-\ln b\right) ^{2}\min \left\{ a,b\right\} \leq \frac{%
a+b}{2}-\sqrt{ab}\leq \frac{1}{8}\left( \ln a-\ln b\right) ^{2}\max \left\{
a,b\right\}  \label{e.2.15}
\end{equation}%
and%
\begin{equation}
\exp \left[ \frac{1}{8}\frac{\left( b-a\right) ^{2}}{\max^{2}\left\{
a,b\right\} }\right] \leq \frac{\frac{a+b}{2}}{\sqrt{ab}}\leq \exp \left[ 
\frac{1}{8}\frac{\left( b-a\right) ^{2}}{\min^{2}\left\{ a,b\right\} }\right]
,  \label{e.2.16}
\end{equation}%
for any $a,b>0.$
\end{remark}

Consider the functions 
\begin{equation*}
P_{1}\left( \nu ,x\right) :=\nu \left( 1-\nu \right) \left( x-1\right) \ln x
\end{equation*}%
and 
\begin{equation*}
P_{2}\left( \nu ,x\right) :=\frac{1}{2}\nu \left( 1-\nu \right) \left( \ln
x\right) ^{2}\max \left\{ x,1\right\}
\end{equation*}%
for $\nu \in \left[ 0,1\right] $ and $x>0.$ A $3D$ plot for $\nu \in \left(
0,1\right) $ and $x\in \left( 0,2\right) $ reveals that the difference $%
P_{2}\left( \nu ,x\right) -P_{1}\left( \nu ,x\right) $ takes both positive
and negative values showing that there is no ordering between the upper
bounds of the quantity $\left( 1-\nu \right) a+\nu b-a^{1-\nu }b^{\nu }$
provided by (\ref{e.2.10}) and (\ref{e.2.13}) respectively.

Also, if we consider the functions%
\begin{equation*}
Q_{1}\left( \nu ,x\right) :=\exp \left[ \nu \left( 1-\nu \right) \frac{%
\left( x-1\right) ^{2}}{x}\right]
\end{equation*}%
and%
\begin{equation*}
Q_{2}\left( \nu ,x\right) :=\exp \left[ \frac{1}{2}\nu \left( 1-\nu \right) 
\frac{\left( x-1\right) ^{2}}{\min^{2}\left\{ x,1\right\} }\right]
\end{equation*}%
for $\nu \in \left[ 0,1\right] $ and $x>0,$ then a $3D$ plot for $\nu \in
\left( 0,1\right) $ and $x\in \left( 0,10\right) $ reveals that the
difference $P_{2}\left( \nu ,x\right) -P_{1}\left( \nu ,x\right) $ takes
also both positive and negative values showing that there is no ordering
between the upper bounds of the quantity $\frac{\left( 1-\nu \right) a+\nu b%
}{a^{1-\nu }b^{\nu }}$ provided by (\ref{e.2.11}) and (\ref{e.2.14}).

\section{Operator Inequalities}

Let $A$ be a positive operator and $B$ a selfadjoint operator. Assume that
the spectrum of $A^{-1/2}BA^{-1/2},$ $\limfunc{Sp}\left(
A^{-1/2}BA^{-1/2}\right) $ is included in $I,$ an interval of real numbers
and $f:I\rightarrow \mathbb{R}$ a continuous function on $I.$ Using the
functional calculus for continuous functions we can consider the selfadjoint
operator%
\begin{equation}
A\sharp _{f}B:=A^{1/2}f\left( A^{-1/2}BA^{-1/2}\right) A^{1/2}.
\label{e.3.1}
\end{equation}%
If $f\left( x\right) =x^{\nu }$ with $\nu \in \lbrack 0,1]$ then by (\ref%
{e.3.1}) we recapture the concept of weighted geometric mean of two
operators.

We have the following result:

\begin{theorem}
\label{t.3.1}Let $A,$ $B$ be two positive operators. Then we have%
\begin{equation}
\frac{1}{4}\nu \left( 1-\nu \right) A\sharp _{f_{\min }}B\leq A\nabla _{\nu
}B-A\sharp _{\nu }B\leq \frac{1}{4}\nu \left( 1-\nu \right) A\sharp
_{f_{\max }}B  \label{e.3.2}
\end{equation}%
for any $\nu \in \left[ 0,1\right] ,$ where $f_{\min },$ $f_{\max }:\left(
0,\infty \right) \rightarrow \mathbb{R}$ are defined by%
\begin{equation}
f_{\min }\left( x\right) =\left( x+1-\left\vert x-1\right\vert \right) \ln
^{2}x,\text{ }f_{\max }\left( x\right) =\left( x+1+\left\vert x-1\right\vert
\right) \ln ^{2}x.  \label{e.3.2.a}
\end{equation}%
In particular, we have%
\begin{equation}
\frac{1}{16}A\sharp _{f_{\min }}B\leq A\nabla B-A\sharp B\leq \frac{1}{16}%
A\sharp _{f_{\max }}B.  \label{e.3.3}
\end{equation}
\end{theorem}

\begin{proof}
From the inequality (\ref{e.2.13}) we have%
\begin{align}
\frac{1}{2}\nu \left( 1-\nu \right) \min \left\{ 1,x\right\} \ln ^{2}x& \leq
1-\nu +\nu x-x^{\nu }  \label{e.3.4} \\
& \leq \frac{1}{2}\nu \left( 1-\nu \right) \max \left\{ 1,x\right\} \ln ^{2}x
\notag
\end{align}%
for any $x>0$ and $\nu \in \left[ 0,1\right] .$

Since $\min \left\{ 1,x\right\} =\frac{1}{2}\left( x+1-\left\vert
x-1\right\vert \right) $ and $\max \left\{ 1,x\right\} =\frac{1}{2}\left(
x+1+\left\vert x-1\right\vert \right) $ then (\ref{e.3.4}) can be written as%
\begin{align}
& \frac{1}{4}\nu \left( 1-\nu \right) \left( x+1-\left\vert x-1\right\vert
\right) \ln ^{2}x  \label{e.3.5} \\
& \leq 1-\nu +\nu x-x^{\nu }  \notag \\
& \leq \frac{1}{4}\nu \left( 1-\nu \right) \left( x+1+\left\vert
x-1\right\vert \right) \ln ^{2}x  \notag
\end{align}%
for any $x>0$ and $\nu \in \left[ 0,1\right] .$

Using the functional calculus for continuous functions we have for any
positive $X$ that%
\begin{align*}
& \frac{1}{4}\nu \left( 1-\nu \right) \left( X+I-\left\vert X-I\right\vert
\right) \ln ^{2}X \\
& \leq \left( 1-\nu \right) I+\nu X-X^{\nu } \\
& \leq \frac{1}{4}\nu \left( 1-\nu \right) \left( X+I+\left\vert
X-I\right\vert \right) \ln ^{2}X
\end{align*}%
where $I$ is the identity operator.

Substituting $A^{-1/2}BA^{-1/2}$ for $X$ we have%
\begin{align}
& \frac{1}{4}\nu \left( 1-\nu \right) \left( A^{-1/2}BA^{-1/2}+I-\left\vert
A^{-1/2}BA^{-1/2}-I\right\vert \right)   \label{e.3.6} \\
& \times \ln ^{2}\left( A^{-1/2}BA^{-1/2}\right)   \notag \\
& \leq \left( 1-\nu \right) I+\nu A^{-1/2}BA^{-1/2}-\left(
A^{-1/2}BA^{-1/2}\right) ^{\nu }  \notag \\
& \leq \frac{1}{4}\nu \left( 1-\nu \right) \left(
A^{-1/2}BA^{-1/2}+I-\left\vert A^{-1/2}BA^{-1/2}-I\right\vert \right)  
\notag \\
& \times \ln ^{2}\left( A^{-1/2}BA^{-1/2}\right)   \notag
\end{align}%
for any $\nu \in \left[ 0,1\right] .$

Multiplying both sides of (\ref{e.3.6}) by $A^{1/2}$ we get the desired
result (\ref{e.3.2}).
\end{proof}

The following result provides simpler lower and upper bounds for the
difference between the weighted arithmetic and geometric operator means.

\begin{theorem}
\label{t.3.2}Let $A,$ $B$ be two positive operators such that 
\begin{equation*}
0<kI\leq A^{-1/2}BA^{-1/2}\leq KI,
\end{equation*}%
for some constants $k,$ $K.$ Then we have 
\begin{equation}
\frac{1}{4}\nu \left( 1-\nu \right) \min_{x\in \left[ k,K\right] }f_{\min
}\left( x\right) A\leq A\nabla _{\nu }B-A\sharp _{\nu }B\leq \frac{1}{4}\nu
\left( 1-\nu \right) \max_{x\in \left[ k,K\right] }f_{\max }\left( x\right) A
\label{e.3.7}
\end{equation}%
for any $\nu \in \left[ 0,1\right] ,$ where $f_{\min },$ $f_{\max }$ are
defined by (\ref{e.3.2.a}).
\end{theorem}

\begin{proof}
From the inequality (\ref{e.3.5}) we have%
\begin{align}
\frac{1}{4}\nu \left( 1-\nu \right) \min_{x\in \left[ k,K\right] }f_{\min
}\left( x\right) & \leq 1-\nu +\nu x-x^{\nu }  \label{e.3.8} \\
& \leq \frac{1}{4}\nu \left( 1-\nu \right) \max_{x\in \left[ k,K\right]
}f_{\max }\left( x\right)   \notag
\end{align}%
for any $x\in \left[ k,K\right] $ and $\nu \in \left[ 0,1\right] .$

If $X$ is a selfadjoint operator with $\limfunc{Sp}\left( X\right) \subset %
\left[ k,K\right] ,$ then by (\ref{e.3.8}) we have%
\begin{align}
\frac{1}{4}\nu \left( 1-\nu \right) \min_{x\in \left[ k,K\right] }f_{\min
}\left( x\right) I& \leq \left( 1-\nu \right) I+\nu X-X^{\nu }  \label{e.3.9}
\\
& \leq \frac{1}{4}\nu \left( 1-\nu \right) \max_{x\in \left[ k,K\right]
}f_{\max }\left( x\right) I.  \notag
\end{align}%
for any $\nu \in \left[ 0,1\right] .$

Now, if we take in (\ref{e.3.9}) $X=A^{-1/2}BA^{-1/2},$ then we get%
\begin{align}
& \frac{1}{4}\nu \left( 1-\nu \right) \min_{x\in \left[ k,K\right] }f_{\min
}\left( x\right) I  \label{e.3.10} \\
& \leq \left( 1-\nu \right) I+\nu A^{-1/2}BA^{-1/2}-\left(
A^{-1/2}BA^{-1/2}\right) ^{\nu }  \notag \\
& \leq \frac{1}{4}\nu \left( 1-\nu \right) \max_{x\in \left[ k,K\right]
}f_{\max }\left( x\right) I.  \notag
\end{align}%
Multiplying both sides of (\ref{e.3.2}) by $A^{1/2}$ we get the desired
result (\ref{e.3.7}).
\end{proof}

\begin{remark}
\label{r.3.1}If $0<m^{\prime }I\leq A\leq mI<MI\leq B\leq M^{\prime }I$ for
positive real numbers $m,$ $m^{\prime },$ $M,$ $M^{\prime }$ then by putting 
$h:=\frac{M}{m},$ $h^{\prime }:=\frac{M^{\prime }}{m^{\prime }}$ we have%
\begin{equation*}
0<h^{\prime }I\leq A^{-1/2}BA^{-1/2}\leq hI.
\end{equation*}%
Therefore we can take in Theorem \ref{t.3.2} $k=h^{\prime }$ and $K=h.$

If $0<m^{\prime }I\leq B\leq mI<MI\leq A\leq M^{\prime }I,$ then%
\begin{equation*}
0<\frac{1}{h}I\leq A^{-1/2}BA^{-1/2}\leq \frac{1}{h^{\prime }}I
\end{equation*}%
and we can take in Theorem \ref{t.3.2} $k=\frac{1}{h}$ and $K=\frac{1}{%
h^{\prime }}.$
\end{remark}

The following multiplicative refinement and reverse of Young's operator
inequality also holds.

\begin{theorem}
\label{t.3.3}Let $A,$ $B$ be two positive operators such that 
\begin{equation*}
0<kI\leq A^{-1/2}BA^{-1/2}\leq KI,
\end{equation*}%
for some constants $k,K.$ Then we have%
\begin{multline}
\exp \left[ \frac{1}{2}\nu \left( 1-\nu \right) \left( 1-\frac{\min \left\{
1,K\right\} }{\max \left\{ 1,k\right\} }\right) ^{2}\right] A\sharp _{\nu
}B\leq A\nabla _{\nu }B  \label{e.3.11} \\
\leq \exp \left[ \frac{1}{2}\nu \left( 1-\nu \right) \left( \frac{\max
\left\{ 1,K\right\} }{\min \left\{ 1,k\right\} }-1\right) ^{2}\right]
A\sharp _{\nu }B
\end{multline}%
any $\nu \in \left[ 0,1\right] .$
\end{theorem}

\begin{proof}
From the inequality (\ref{e.2.12.b}) we have 
\begin{align}
& \exp \left[ \frac{1}{2}\nu \left( 1-\nu \right) \left( 1-\frac{\min
\left\{ 1,x\right\} }{\max \left\{ 1,x\right\} }\right) ^{2}\right] x^{\nu }
\label{e.3.12} \\
& \leq 1-\nu +\nu x  \notag \\
& \leq \exp \left[ \frac{1}{2}\nu \left( 1-\nu \right) \left( \frac{\max
\left\{ 1,x\right\} }{\min \left\{ 1,x\right\} }-1\right) ^{2}\right] x^{\nu
}  \notag
\end{align}%
for any $x>0$ and any $\nu \in \left[ 0,1\right] .$

If $x\in \left[ k,K\right] \subset \left( 0,\infty \right) $ then 
\begin{equation*}
0\leq \frac{\max \left\{ 1,x\right\} }{\min \left\{ 1,x\right\} }-1\leq 
\frac{\max \left\{ 1,K\right\} }{\min \left\{ 1,k\right\} }-1
\end{equation*}%
and%
\begin{equation*}
0\leq 1-\frac{\min \left\{ 1,K\right\} }{\max \left\{ 1,k\right\} }\leq 1-%
\frac{\min \left\{ 1,x\right\} }{\max \left\{ 1,x\right\} },
\end{equation*}%
which implies that%
\begin{equation*}
\exp \left[ \frac{1}{2}\nu \left( 1-\nu \right) \left( \frac{\max \left\{
1,x\right\} }{\min \left\{ 1,x\right\} }-1\right) ^{2}\right] \leq \exp %
\left[ \frac{1}{2}\nu \left( 1-\nu \right) \left( \frac{\max \left\{
1,K\right\} }{\min \left\{ 1,k\right\} }-1\right) ^{2}\right]
\end{equation*}%
and%
\begin{equation*}
\exp \left[ \frac{1}{2}\nu \left( 1-\nu \right) \left( 1-\frac{\min \left\{
1,K\right\} }{\max \left\{ 1,k\right\} }\right) ^{2}\right] \leq \exp \left[ 
\frac{1}{2}\nu \left( 1-\nu \right) \left( 1-\frac{\min \left\{ 1,x\right\} 
}{\max \left\{ 1,x\right\} }\right) ^{2}\right] .
\end{equation*}%
By (\ref{e.3.12}) we then have%
\begin{align}
& \exp \left[ \frac{1}{2}\nu \left( 1-\nu \right) \left( 1-\frac{\min
\left\{ 1,K\right\} }{\max \left\{ 1,k\right\} }\right) ^{2}\right] x^{\nu }
\label{e.3.13} \\
& \leq 1-\nu +\nu x  \notag \\
& \leq \exp \left[ \frac{1}{2}\nu \left( 1-\nu \right) \left( \frac{\max
\left\{ 1,K\right\} }{\min \left\{ 1,k\right\} }-1\right) ^{2}\right] x^{\nu
}  \notag
\end{align}%
for any $x\in \left[ k,K\right] $ and any $\nu \in \left[ 0,1\right] .$

If $X$ is a selfadjoint operator with $\limfunc{Sp}\left( X\right) \subset %
\left[ k,K\right] ,$ then by (\ref{e.3.13}) we have%
\begin{align}
& \exp \left[ \frac{1}{2}\nu \left( 1-\nu \right) \left( 1-\frac{\min
\left\{ 1,K\right\} }{\max \left\{ 1,k\right\} }\right) ^{2}\right] X^{\nu }
\label{e.3.14} \\
& \leq \left( 1-\nu \right) I+\nu X  \notag \\
& \leq \exp \left[ \frac{1}{2}\nu \left( 1-\nu \right) \left( \frac{\max
\left\{ 1,K\right\} }{\min \left\{ 1,k\right\} }-1\right) ^{2}\right] X^{\nu
}.  \notag
\end{align}%
Now, if we take in (\ref{e.3.14}) $X=A^{-1/2}BA^{-1/2},$ then we get%
\begin{align}
& \exp \left[ \frac{1}{2}\nu \left( 1-\nu \right) \left( 1-\frac{\min
\left\{ 1,K\right\} }{\max \left\{ 1,k\right\} }\right) ^{2}\right] \left(
A^{-1/2}BA^{-1/2}\right) ^{\nu }  \label{e.3.15} \\
& \leq \left( 1-\nu \right) I+\nu A^{-1/2}BA^{-1/2}  \notag \\
& \leq \exp \left[ \frac{1}{2}\nu \left( 1-\nu \right) \left( \frac{\max
\left\{ 1,K\right\} }{\min \left\{ 1,k\right\} }-1\right) ^{2}\right] \left(
A^{-1/2}BA^{-1/2}\right) ^{\nu }.  \notag
\end{align}%
By multiplying both sides of (\ref{e.3.6}) by $A^{1/2}$ we get the desired
result (\ref{e.3.11}).
\end{proof}

\begin{corollary}
\label{c.3.1}If either $0<m^{\prime }I\leq A\leq mI<MI\leq B\leq M^{\prime
}I $ for positive real numbers $m,$ $m^{\prime },$ $M,$ $M^{\prime }$ or $%
0<m^{\prime }I\leq B\leq mI<MI\leq A\leq M^{\prime }I,$ then by putting $h:=%
\frac{M}{m},$ $h^{\prime }:=\frac{M^{\prime }}{m^{\prime }}$ we have%
\begin{align}
\exp \left[ \frac{1}{2}\nu \left( 1-\nu \right) \left( \frac{h^{\prime }-1}{%
h^{\prime }}\right) ^{2}\right] A\sharp _{\nu }B& \leq A\nabla _{\nu }B
\label{e.3.16} \\
& \leq \exp \left[ \frac{1}{2}\nu \left( 1-\nu \right) \left( h-1\right) ^{2}%
\right] A\sharp _{\nu }B.  \notag
\end{align}
\end{corollary}

\begin{proof}
If $0<m^{\prime }I\leq A\leq mI<MI\leq B\leq M^{\prime }I,$ then we have%
\begin{equation*}
0<h^{\prime }I\leq A^{-1/2}BA^{-1/2}\leq hI.
\end{equation*}%
Using (\ref{e.3.11}) for $k=h^{\prime }$ and $K=h$ we get 
\begin{multline*}
\exp \left[ \frac{1}{2}\nu \left( 1-\nu \right) \left( 1-\frac{1}{h^{\prime }%
}\right) ^{2}\right] A\sharp _{\nu }B\leq A\nabla _{\nu }B \\
\leq \exp \left[ \frac{1}{2}\nu \left( 1-\nu \right) \left( h-1\right) ^{2}%
\right] A\sharp _{\nu }B
\end{multline*}%
and the inequality (\ref{e.3.16}) is proved.

If $0<m^{\prime }I\leq B\leq mI<MI\leq A\leq M^{\prime }I,$ then we have%
\begin{equation*}
0<\frac{1}{h}I\leq A^{-1/2}BA^{-1/2}\leq \frac{1}{h^{\prime }}I.
\end{equation*}%
Using (\ref{e.3.11}) for $k=\frac{1}{h}$ and $K=\frac{1}{h^{\prime }}$ we get%
\begin{multline*}
\exp \left[ \frac{1}{2}\nu \left( 1-\nu \right) \left( 1-\frac{1}{h^{\prime }%
}\right) ^{2}\right] A\sharp _{\nu }B\leq A\nabla _{\nu }B \\
\leq \exp \left[ \frac{1}{2}\nu \left( 1-\nu \right) \left( h-1\right) ^{2}%
\right] A\sharp _{\nu }B
\end{multline*}%
and the inequality (\ref{e.3.16}) is also obtained.
\end{proof}

\section{Some Extension for Functions}

We can extend some of the above results for functions of operators as
follows.

\begin{theorem}
\label{t.4.1}Let $f:J\subset \mathbb{R}\rightarrow \mathbb{R}$ be a twice
differentiable function on the interval $\mathring{J}$, the interior of $J$.
Suppose that there exists the constants $d,$ $D$ such that%
\begin{equation}
d\leq f^{\prime \prime }\left( t\right) \leq D\text{ for any }t\in \mathring{%
J}.  \label{J}
\end{equation}%
If $A$ is a positive operator and $B$ a selfadjoint operator such that%
\begin{equation}
\gamma I\leq A^{-1/2}BA^{-1/2}\leq \Gamma I,  \label{C}
\end{equation}%
with $\left[ \gamma ,\Gamma \right] \subset \mathring{J},$ then we have%
\begin{align}
& \frac{1}{2}\left( \Gamma A^{1/2}-BA^{-1/2}\right) \left(
A^{-1/2}B-A^{1/2}\gamma \right) d  \label{e.4.1} \\
& \leq \frac{1}{\Gamma -\gamma }\left[ \left( \Gamma A-B\right) f\left(
\gamma \right) +\left( B-A\gamma \right) f\left( \Gamma \right) \right]
-A\sharp _{f}B  \notag \\
& \leq \frac{1}{2}\left( \Gamma A^{1/2}-BA^{-1/2}\right) \left(
A^{-1/2}B-A^{1/2}\gamma \right) D.  \notag
\end{align}
\end{theorem}

\begin{proof}
From Lemma \ref{l.2.1} we have%
\begin{align}
\frac{1}{2}\nu \left( 1-\nu \right) d\left( \Gamma -\gamma \right) ^{2}&
\leq \left( 1-\nu \right) f\left( \gamma \right) +\nu f\left( \Gamma \right)
-f\left( \left( 1-\nu \right) \gamma +\nu \Gamma \right)   \label{e.4.2} \\
& \leq \frac{1}{2}\nu \left( 1-\nu \right) D\left( \Gamma -\gamma \right)
^{2}  \notag
\end{align}%
for any $\nu \in \left[ 0,1\right] .$

If we take in (\ref{e.4.2}) $\nu =\frac{x-\gamma }{\Gamma -\gamma }\in \left[
0,1\right] $ for $x\in \left[ \gamma ,\Gamma \right] ,$ then we get%
\begin{eqnarray*}
\frac{1}{2}\left( \Gamma -x\right) \left( x-\gamma \right) d &\leq &\frac{1}{%
\Gamma -\gamma }\left[ \left( \Gamma -x\right) f\left( \gamma \right)
+\left( x-\gamma \right) f\left( \Gamma \right) \right] -f\left( x\right) \\
&\leq &\frac{1}{2}\left( \Gamma -x\right) \left( x-\gamma \right) D
\end{eqnarray*}%
for any $x\in \left[ \gamma ,\Gamma \right] .$

Using the functional calculus for continuous functions we have%
\begin{align}
& \frac{1}{2}\left( \Gamma I-X\right) \left( X-\gamma I\right) d
\label{e.4.3} \\
& \leq \frac{1}{\Gamma -\gamma }\left[ \left( \Gamma I-X\right) f\left(
\gamma \right) +\left( X-\gamma I\right) f\left( \Gamma \right) \right]
-f\left( X\right)  \notag \\
& \leq \frac{1}{2}\left( \Gamma I-X\right) \left( X-\gamma I\right) D  \notag
\end{align}%
for any selfadjoint operator $X$ with $\limfunc{Sp}(X)\subset \left[ \gamma
,\Gamma \right] .$

Now, if we write the inequality (\ref{e.4.3}) for $X$ $=A^{-1/2}BA^{-1/2}$
then we get%
\begin{align}
& \frac{1}{2}\left( \Gamma I-A^{-1/2}BA^{-1/2}\right) \left(
A^{-1/2}BA^{-1/2}-\gamma I\right) d  \label{e.4.4} \\
& \leq \frac{1}{\Gamma -\gamma }\left[ \left( \Gamma
I-A^{-1/2}BA^{-1/2}\right) f\left( \gamma \right) +\left(
A^{-1/2}BA^{-1/2}-\gamma I\right) f\left( \Gamma \right) \right]   \notag \\
& -f\left( A^{-1/2}BA^{-1/2}\right)   \notag \\
& \leq \frac{1}{2}\left( \Gamma I-A^{-1/2}BA^{-1/2}\right) \left(
A^{-1/2}BA^{-1/2}-\gamma I\right) D.  \notag
\end{align}%
By multiplying both sides of (\ref{e.4.4}) by $A^{1/2}$ we get (\ref{e.4.1}).
\end{proof}

If $\gamma >0$ in (\ref{C}) and take $f\left( t\right) =t^{p}$, $t>0$ where $%
p\in \left( 0,\infty \right) \cup (1,\infty )$ then%
\begin{equation*}
f^{\prime \prime }\left( t\right) =p\left( p-1\right) t^{p-2},t>0
\end{equation*}%
and if we take%
\begin{equation*}
d=d_{p}:=p\left( p-1\right) \left\{ 
\begin{array}{l}
\gamma ^{p-2},\text{ }p\geq 2 \\ 
\\ 
\Gamma ^{p-2},p\in \left( 0,\infty \right) \cup (1,2)%
\end{array}%
\right. 
\end{equation*}%
and%
\begin{equation*}
D=D_{p}:=p\left( p-1\right) \left\{ 
\begin{array}{l}
\Gamma ^{p-2},\text{ }p\geq 2 \\ 
\\ 
\gamma ^{p-2},p\in \left( 0,\infty \right) \cup (1,2)%
\end{array}%
\right. 
\end{equation*}%
we have from (\ref{e.3.13}) that%
\begin{align}
& \frac{1}{2}\left( \Gamma A^{1/2}-BA^{-1/2}\right) \left(
A^{-1/2}B-A^{1/2}\gamma \right) d_{p}  \label{e.4.5} \\
& \leq \frac{\Gamma A-B}{\Gamma -\gamma }\gamma ^{p}+\frac{B-A\gamma }{%
\Gamma -\gamma }\Gamma ^{p}-A\sharp _{p}B  \notag \\
& \leq \frac{1}{2}\left( \Gamma A^{1/2}-BA^{-1/2}\right) \left(
A^{-1/2}B-A^{1/2}\gamma \right) D_{p}  \notag
\end{align}%
where $A\sharp _{p}B:=A^{1/2}\left( A^{-1/2}BA^{-1/2}\right) ^{p}A^{1/2},$ $%
p\in \left( 0,\infty \right) \cup (1,\infty )$ and $A,$ $B>0$ satisfy the
condition (\ref{C}).

If $p\in \left( 0,1\right) $ and $A,$ $B>0$ satisfy the condition (\ref{C})
with $\gamma >0$ then by taking $f\left( t\right) =-t^{p}$, $t>0$ we have $%
f^{\prime \prime }\left( t\right) =p\left( 1-p\right) t^{p-2},$ $t>0$ giving
that 
\begin{equation*}
p\left( 1-p\right) \Gamma ^{p-2}\leq f^{\prime \prime }\left( t\right) \leq
p\left( 1-p\right) \gamma ^{p-2}\text{ for }t\in \left[ \gamma ,\Gamma %
\right] .
\end{equation*}%
Therefore, by choosing $d=p\left( 1-p\right) \Gamma ^{p-2}$ and $D=p\left(
1-p\right) \gamma ^{p-2}$ in (\ref{e.3.13}) we get%
\begin{align}
& p\left( 1-p\right) \Gamma ^{p-2}\left( \Gamma A^{1/2}-BA^{-1/2}\right)
\left( A^{-1/2}B-A^{1/2}\gamma \right)   \label{e.4.6} \\
& \leq A\sharp _{p}B-\frac{\Gamma A-B}{\Gamma -\gamma }\gamma ^{p}-\frac{%
B-A\gamma }{\Gamma -\gamma }\Gamma ^{p}  \notag \\
& \leq p\left( 1-p\right) \gamma ^{p-2}\left( \Gamma
A^{1/2}-BA^{-1/2}\right) \left( A^{-1/2}B-A^{1/2}\gamma \right)   \notag
\end{align}%
provided $p\in \left( 0,1\right) $ and $A,$ $B>0$ satisfy the condition (\ref%
{C}) with $\gamma >0.$

\begin{theorem}
\label{t.4.2}Let $f:J\subset \mathbb{R}\rightarrow \mathbb{R}$ be a twice
differentiable function on the interval $\mathring{J}$, the interior of $J$.
Suppose that there exists the constants $d,$ $D$ such that (\ref{J}) is
valid. If $A$ is a positive operator and $B$ a selfadjoint operator such
that the condition (\ref{C}) is valid with $\left[ \gamma ,\Gamma \right]
\subset \mathring{J}$ and $\gamma <1<\Gamma $ then we have%
\begin{multline}
\frac{1}{2}\nu \left( 1-\nu \right) dA^{1/2}\left(
A^{-1/2}BA^{-1/2}-I\right) ^{2}A^{1/2}  \label{e.4.7} \\
\leq \left( 1-\nu \right) f\left( 1\right) A+\nu A^{1/2}f\left(
A^{-1/2}BA^{-1/2}\right) A^{1/2} \\
-A^{1/2}f\left( \left( 1-\nu \right) I+\nu A^{-1/2}BA^{-1/2}\right) A^{1/2}
\\
\leq \frac{1}{2}\nu \left( 1-\nu \right) DA^{1/2}\left(
A^{-1/2}BA^{-1/2}-I\right) ^{2}A^{1/2}
\end{multline}%
for any $\nu \in \left[ 0,1\right] .$

In particular, we have%
\begin{multline}
\frac{1}{8}dA^{1/2}\left( A^{-1/2}BA^{-1/2}-I\right) ^{2}A^{1/2}
\label{e.4.7.a} \\
\leq \frac{1}{2}\left[ f\left( 1\right) A+A^{1/2}f\left(
A^{-1/2}BA^{-1/2}\right) A^{1/2}\right]  \\
-A^{1/2}f\left( \frac{I+A^{-1/2}BA^{-1/2}}{2}\right) A^{1/2} \\
\leq \frac{1}{8}DA^{1/2}\left( A^{-1/2}BA^{-1/2}-I\right) ^{2}A^{1/2}.
\end{multline}
\end{theorem}

\begin{proof}
We have from (\ref{e.2.1}) that 
\begin{align}
\frac{1}{2}\nu \left( 1-\nu \right) d\left( b-a\right) ^{2}& \leq \left(
1-\nu \right) f\left( a\right) +\nu f\left( b\right) -f\left( \left( 1-\nu
\right) a+\nu b\right)   \label{e.4.8} \\
& \leq \frac{1}{2}\nu \left( 1-\nu \right) D\left( b-a\right) ^{2}  \notag
\end{align}%
for any $a,$ $b\in \left[ \gamma ,\Gamma \right] $ and $\nu \in \left[ 0,1%
\right] .$

If we take  $a=1$ and $b=x$ in (\ref{e.4.8}) then we get%
\begin{align}
\frac{1}{2}\nu \left( 1-\nu \right) d\left( x-1\right) ^{2}& \leq \left(
1-\nu \right) f\left( 1\right) +\nu f\left( x\right) -f\left( \left( 1-\nu
\right) 1+\nu x\right)   \label{e.4.9} \\
& \leq \frac{1}{2}\nu \left( 1-\nu \right) D\left( x-1\right) ^{2}  \notag
\end{align}%
for any $x\in \left( \gamma ,\Gamma \right) .$

This implies in the operator order that%
\begin{align}
\frac{1}{2}\nu \left( 1-\nu \right) d\left( X-I\right) ^{2}& \leq \left(
1-\nu \right) f\left( 1\right) I+\nu f\left( X\right) -f\left( \left( 1-\nu
\right) I+\nu X\right)   \label{e.4.10} \\
& \leq \frac{1}{2}\nu \left( 1-\nu \right) D\left( X-I\right) ^{2}  \notag
\end{align}%
for any selfadjoint operator $X$ with $\limfunc{Sp}(X)\subset \left[ \gamma
,\Gamma \right] .$

If we take in (\ref{e.4.10}) $X=A^{-1/2}BA^{-1/2}$ and multiply in both
sides with $A^{1/2}$ then we get the desired result (\ref{e.4.7}).
\end{proof}

Consider the convex functions $f\left( t\right) =-\ln t,$ $t>0.$ Then $%
f^{\prime \prime }\left( t\right) =\frac{1}{t^{2}},$ $t\in \left[ \gamma
,\Gamma \right] $ and by taking $d=\frac{1}{\Gamma ^{2}}$ and $D=\frac{1}{%
\gamma ^{2}}$ in (\ref{e.4.7}) we get%
\begin{multline}
\frac{1}{2\Gamma ^{2}}\nu \left( 1-\nu \right) A^{1/2}\left(
A^{-1/2}BA^{-1/2}-I\right) ^{2}A^{1/2}  \label{e.3.23} \\
\leq A^{1/2}\left[ \ln \left( \left( 1-\nu \right) I+\nu
A^{-1/2}BA^{-1/2}\right) \right] A^{1/2} \\
-\nu A^{1/2}\left[ \ln \left( A^{-1/2}BA^{-1/2}\right) \right] A^{1/2} \\
\leq \frac{1}{2\gamma ^{2}}\nu \left( 1-\nu \right) A^{1/2}\left(
A^{-1/2}BA^{-1/2}-I\right) ^{2}A^{1/2}
\end{multline}%
provided $A,$ $B>0$ and satisfy (\ref{C}) while $\nu \in \left[ 0,1\right] .$

In particular, we have%
\begin{multline}
\frac{1}{8\Gamma ^{2}}A^{1/2}\left( A^{-1/2}BA^{-1/2}-I\right) ^{2}A^{1/2}
\label{e.2.24} \\
\leq A^{1/2}\left[ \ln \left( \frac{I+A^{-1/2}BA^{-1/2}}{2}\right) \right]
A^{1/2}-\frac{1}{2}A^{1/2}\left[ \ln \left( A^{-1/2}BA^{-1/2}\right) \right]
A^{1/2} \\
\leq \frac{1}{8\gamma ^{2}}A^{1/2}\left( A^{-1/2}BA^{-1/2}-I\right)
^{2}A^{1/2}.
\end{multline}

\end{document}